\documentclass{article} 
\usepackage{cite}
\usepackage{amsmath, tikz-cd}
\usepackage{multicol}
\usepackage{subfigure}
\usepackage{amssymb}
\usepackage{enumitem}
\usepackage{amsthm}
\usepackage{hyperref} 

\usepackage{cite}

\theoremstyle{theorem}

\theoremstyle{definition}

\numberwithin{equation}{section}
\theoremstyle{plain}

\newtheorem*{Corollary*}{Corollary}
\newtheorem{Theorem}[equation]{Theorem}
\newtheorem*{Theorem*}{Theorem}

\theoremstyle{definition}
\newtheorem{Definition}[equation]{Definition}

\allowdisplaybreaks

\makeatletter
\@addtoreset{footnote}{page}
\makeatother

\begin{document}

\title{How to Bake a Composition Operator from Scratch}

\author{\\ Michael R Pilla \\           
\scriptsize Florida Polytechnic University\\    }                 

\maketitle

\begin{abstract}
At face value, composition operators look like something first experienced in algebra. It's just composition of functions after all. However, there is a deep and technical setting in which these objects live, behaving like infinite-dimensional matrices acting on strange yet fascinating spaces. We take a novel approach to introducing a dense mathematical topic by mirroring that of a so-called scratch kitchen, where ingredients are not prepackaged, processed, or frozen but rather are fresh. While some mathematical maturity will certainly help in the same way that baking experience helps with a new recipe, nothing in one's background is assumed. Along the way, we make no reservations about making historical or philosophical comments to spark reader's interest as well. Bon app\'etit!
\end{abstract}

\section{A WORD ON MOTIVATION}

\begin{quote}
 ``{\it If you wish to make an apple pie from scratch, you must first invent the universe.} - Carl Sagan
\end{quote} 

I once heard a story about a boy who wished to have great strength. He lived on a farm which raised bulls amongst other animals. One day he decided that he wanted to be strong enough to lift a bull. He decisively made his plan. The next time a bull was to be born on their lot, he would lift it every day. As the baby bull would grow, so would his strength. One day, it stood to reason, when the bull reaches full maturity, the boy would find himself capable of lifting the bull! It was not long until such a newborn arrived. The boy, eager to begin his journey to great strength, swiftly approached the recently delivered baby. However, as he went to complete his first day's task, he realized that the newborn was already too heavy for him to lift!

Even the most well-written expositions will fail in their execution if the reader cannot pick up the first definitions. Additionally, for the sake of brevity, an array of assumptions are made by authors of mathematical content about the backround knowledge of the intended readers. There is no sort of standard about these assumptions beyond the wildy varying beliefs of the authors who all too often just assume their diverse audience had the mathematical training and resources parallel to their own experience. A common retort is that the student `ought' to have had such and such a background. But alas, reality has a track record of failing to align with expectactions, a lesson we all would do well to better acknowledge. Of course, now we live in an era where any definitions we desire are at our fingertips. However as a literature professor once taught me, one ought not interrupt the flow of a good reading. There is value and consistency in self-containment. Additionally, when seeking out external information, the reader is left to her own discretion as to decide which resource to attend, not only breaking the flow of the reading, but also leaving the reader to burn energy trying to decipher explanations housed in an oft foreign context with a focus on hair-splitting correctness over pedagogy (I'm looking at you, Wikipedia). Good exposition and the fleshing out of details takes great work. And it's not just that the author is too lazy to include such information, but rather that our profession highly values terseness and honestly, punishes exposition, pressuring the author to exclude details that might otherwise hook a much larger audience into gaining interest in the topic. This is to the detriment of the mathematical community.

As an experiment of concept, we shall embark on a story about a mathematical object known as a composition operator. We will introduce the topic from first principles. Now, in reality, no one really bakes from scratch. The sugar, the vanilla extract, and so forth, have been, in fact, purchased from the store. Perhaps one grew the apples from a backyard apple tree or collected mint from a backyard herb garden. At some point, though, one either depends on basic store-bought ingredients or the miracle of nature. Creating the universe is probably beyond any mortal's scope. Hence, we will start with ingredients that are reasonably basic enough to be considered scratch, and proceed from there.

Before taking stock of our ingredients, let's chat a bit about what it is that we are baking. Every algebra student (hopefully!) learns about function composition. Given two functions $f$ and $g$ and all values $x$ such that $x$ is in the domain of $g$ and $g(x)$ is in the domain of $f$, we define $f$ composed with $g$ by $f \circ g \equiv f(g(x))$. While I find my students frequently fumble function composition questions involving tricky domains, there's usually a very quick understanding of the actual composition. In fact, composition is a particularly interesting feature of functions that is not shared with numbers. Just as with numbers, one can do arithmetic with functions via pointwise definitions. Thus, for example, one can meaningfully talk about addition $f+g$ by defining it pointwise as $(f+g)(x)=f(x)+g(x)$ for all $x$ in the domain of $f$ and $g$. However, function composition has no analog with numbers. You can't nestle a $2$ inside of a $7$! I know, I know. You're thinking, we haven't even defined a function yet. We're not baking from scratch! Well, we haven't started baking yet. Right now we are talking about what we are going to bake. Sounds delicious!

Thus, composition operators probably have something to do with function composition. However, there is much more to them. Like any mathematical object worth its weight in salt, composition operators are housed in a large well-defined mathematical structure, rich with content and beauty. Intuitively, they can be thought of as a specific type of infinite-dimensional matrix, acting on certain types of vectors space. Our focus will be one of the most important of these, the Hardy space. Now, before learning about function composition, one has to learn about the arithmetic of numbers and usually some analytic geometry. This is the arena where functions perform, the context in which function composition resides. The context we will introduce, however, will be more sophisticated. 

Hence, the starting point of our understanding of composition operators is an understanding of the Hardy space. I recall my first introduction to the space. True to terse mathematical form, within five sentences of the book we arrive at our definition upon which the entire book (and my thesis!) rests. It looked like this (which, at the present time of writing this, looks almost identical to Wikipedia's definition).

\begin{Definition}
Let $f$ be a function analytic in the unit disk $\mathbb{D}$ and $f_r(e^{i\theta})=f(re^{i\theta})$ where $0<r<1$. Let $0<p<\infty$. We define the {\bf Hardy space} $H^p(\mathbb{D})$ to be the set of functions such that 

\begin{equation}
\sup_{0<r<1} \int_0^{2\pi} |f_r(e^{i\theta})|^p \frac{d\theta}{2\pi}<\infty.
\end{equation}

\end{Definition}

No set up. No motivation. Just a definition freely floating in the void. To the informed, subconscious connections are made and perhaps it's description is obvious. For the uninitiated, without an intuition of function spaces, this is a bolt out of the blue.  And such bolts are certainly not restricted to graduate school. I recall the first day in my undergraduate topology course. The professor walked up to the whiteboard and wrote down the definition of a topology, then proceeded to start working out examples of what is and isn't a topology. Why {\it that} definition? What motivated it? I am firmly of the opinion that doctor's orders are followed more closely if the doctor tells the patient precisely {\it why} such action must be taken.

It is the job of the working mathematician to motivate and communicate foundational ideas in the field. Now, it is not reasonable to start every subject matter from scratch. That is what prerequisites are for. But there has to be a pedagogical utility in a brief review of the thread of ideas and concepts that build to the important definitions of a subject, making the material more accesible to a broader audience and inspiring interest in the topics. I can think of no better way to propagate interest in one's field.

\section{PREPARING THE INGREDIENTS}

\subsection{God's creation}

The most basic standard ingredients one can buy at the mathematical supermarket are sets. Even though it's foundational security is blurry, lying at the intersection of cognitive perception and mathematical reality, like zooming into an electron only to see a smeared probability cloud, set theory is still the {\it de facto} practicing foundation for mathematics, fuzzy though it may be, providing us with our most fundamental ingredients. Sets are comprised of elements. What are elements? You know what elements are. Points are elements of the plane. $7$ is an element of the natural numbers. We are not going to define elements in this article. They are considered a primitive notion, like picking mint from my herb garden. This is where the buck stops. Usually there is an unspoken assumption at this point that we share the same biological hardware and what I'm thinking of when I imagine a set, mirrors whatever it is that your cognition has contrived. Naively, you collect them in practically any way you'd like and you have a set. Although if one is {\it too} cavalier about the whole ordeal, one might end up with a paradox. The most notorious example being the infamous Russell's paradox. This will take us too far afield, however. If sets and elements are the nouns that are hardwired into our brains, then operations are the verbs, the actions, the motions in space and time. We will primarily be concerned with binary operations. That is, operations that take two elements in our set and map them to another element in our set. Usually the collection of elements we care about are numbers and the operations are things like addition and multiplication. We are just going to buy these items from the math market. We also buy axioms in order to have something in which to bake. This brings us to our first definition. Regarding order of operations, we recall the classic mnemonic PEMDAS (parentheses, exponents, multiplication, division, addition, subtraction). 

\begin{Definition}
A {\bf field} is a set $F$ of elements, together with the two binary operations of addition $+$ and multiplication $\cdot$ that satisfy the following axioms. Let $a$, $b$, and $c$ be elements of $F$.

\begin{enumerate}
\item $a+(b+c)=(a+b)+c$  \quad  (Associativity of Addition)
\item $a\cdot (b\cdot c)=(a\cdot b)\cdot c$ \quad   (Associativity of Multiplication)
\item $a+b=b+a$ \quad  (Commutativity of Addition)
\item $a \cdot b = b \cdot a$ \quad   (Commutativity of Multiplication)
\item There exist $0$ in $F$ such that $a+0=a$ for all $a$ in $F$ \quad (Additive Identity)
\item There exist $1$ in $F$ such that $a \cdot 1=a$ for all $a$ in $F$ \quad (Multiplicitive Identity)
\item There exist $-a$ in $F$ such that $a+(-a)=0$ for all $a$ in $F$ \quad (Additive Inverse)
\item There exist $a^{-1}$ in $F$ such that $a \cdot a^{-1}=1$ for all $a$ in $F$ \quad (Multiplicitive Inverse)
\item $(a+b) \cdot c=a \cdot c + b \cdot c$ \quad (Distributive Property)
\end{enumerate}

\end{Definition} 

The first four properties tell us how to combine elements while the distributive property tells us how the two binary operations interact. Like a lot of mathematical definitions, these rules are motivated by and extracted from our basic understanding of arithmetic that we learn in elementary school. These are the core rules that our everyday numbers follow. Now that we have this explicit list of rules, however, we needn't even consider numbers as our elements. Any set of things that satisfy these properties qualify as a field. Our primary field, however, will be the complex numbers. These are the set of numbers $a+bi$ where $a$ and $b$ are real numbers and $i \equiv \sqrt{-1}$. Let $z=a+bi$ and $w=c+di$ be complex numbers. Assuming knowledge of the arithmetic of real numbers, we define addition and multiplication of complex numbers, respectively, as follows. 

\begin{itemize}
\item $z+w=(a+bi)+(c+di)=(a+c)+(b+d)i$
\item $z \cdot w = (a+bi) \cdot (c+di)=(ac-bd)+(ad-bc)i$
\end{itemize}
where the notation $ab$ denotes standard arithmetic multiplication of real numbers $a$ and $b$. Our additive and multiplicative complex identities are the same as the real number identities $0$ and $1$, respectively. 

Now, the astute reader might notice that I've just pulled a fast one on you. Perhaps, as Leopold Kronecker famously stated, `God created the integers', a perspective as good as any but I've not defined real numbers yet! Real numbers have been used at the very least since the time of Pythagoras when, as the story goes, one of his pupils walked the plank for proving that the square root of two is irrational (i.e. not the ratio of two integers). Historical records suggest they were possibly known in India centuries before Pythagoras. Yet it took until the 19th century via the work of Joseph Betrand, Richard Dedekind, and Georg Cantor to formalize a definition accepted and utilized by the mathematical community. Given the strong intuition for real numbers via the continuity of the real number line, we are going to purchase them wholesale and assume the reader cannot taste the difference. The reader interested in baking real numbers from scratch ought to look up Dedekind cuts or Cauchy sequences in a complete metric space.

The reason for using complex numbers instead of real numbers is that, in many ways, they live in the Goldilocks region of numbers. For example, by the fundamental theorem of algebra, any polynomial of degree $n$ can be factored into a product of $n$ linear factors (including multiplicity). Additionally, the ability to take one derivative of a function of a complex variable permits the ability to take {\it any} number of derivatives. One important repercussion of this is that holomorphic functions, which are defined as complex-valued functions which are complex differentiable in a neighborhood of every point in their domain, can be written as a power series! Functions that can be written locally as a power series are known as analytic functions. Thus, holomorphic functions are equivalent to analytic functions. This behavior is not the case with real numbers.

\subsection{RAPID LEAVENING}

Next, we move onto the definition of a vector space. The first exposure to this definition usually is in an undergraduate course in linear algebra. Like the definition of a field, it is just a list of axioms that dictate how the elements in a set and their designated operations behave and interact. A vector space is motivated by the fact that fields don't give the right structure for `coordinates'. For example, what if my elements are {\it pairs} of elements in a field, such as the real numbers? After all, we live in a three-dimensional world so it is natural to think of pairs or triples of real numbers. How do we do arithmetic on such sets? The correct framework lies in the definition of a vector space. Even the name `space' evokes intuitions of the physical space in which we live. However, one must tread with care. Intuition motivates definitions and axioms, which are quite formal objects, (at least partially) insulated from our subconscious conceptions. While we are formally working with a well-defined structure, assumptions based on our intuitions mustn't be allowed to sneak in, opening the door for potential errors in our work. Like in the case of the definition of a field, after a historical trial-and-error, mathematicians came upon the following definition for a vector space as a fruitful object of study that also provides an optimal framework for housing our intuitions for sets of coordinates.

To bake a vector space, we need an underlying field $F$ and a binary operation for vector addition and scalar multiplication. The idea is that one can add two elements in the vector space and multiply by a scalar (think, e.g., a real number) to stretch or compress the vector by a given amount. If you've learned about vectors before, you may use these as an intuition for the definition.

\begin{Definition}
A {\bf vector space} is a set $V$ of elements, together with a binary operation of addition $+$ and multiplication $\cdot$ that satisfy the following axioms. Let $a$ and $b$ be elements of $F$ and let $u$, $v$, and $w$ be elements of $V$. We call elements of $V$ {\bf vectors}.

\begin{enumerate}
\item $u+(v+w)=(u+v)+w$  \quad  (Associativity of Vector Addition)
\item $u+v=v+u$ \quad  (Commutativity of Vector Addition)
\item There exist $0$ in $V$ such that $u+0=u$ for all $u$ in $V$ \quad (Additive Identity)
\item There exist $-u$ in $V$ such that $u+(-u)=0$ for all $u$ in $V$ \quad (Additive Inverse)
\item There exist $1$ in $F$ such that $1 \cdot u=u$ for all $u$ in $V$ \quad (Scalar Identity)
\item $a \cdot (b \cdot u)=(a \cdot b) \cdot u$ \quad (Scalar Associativity)
\item  $(a+b) \cdot u=a \cdot u + b \cdot u$ \quad (Scalar Distributive Property)
\item $a \cdot (u+v)=a \cdot u + a \cdot v$ \quad (Vector Distributive Property)
\end{enumerate}

\end{Definition} 

Before moving on, it should be mentioned that mathematics is often rife with implicit assumptions often missed early on. Unless you were paying close attention, it might not be apparent that we have assumed closure properties in this definition. This means that if I combine two vectors via a binary operation such as addition or multiplication, the result is a vector in my set as well. For example, integers have this property. $4+3=7$, which is also an integer. However, positive integers do not. $3-5=-2$ which is not a positive integer.

A natural visualization of a vector space is given by drawing arrows in the rectangular $xy$-plane that is learned in analytic geometry. One places the origin of the arrow at the origin of the $xy$-plane and the end of the arrow at the coordinate given by the vector. Such visualization must be done with care, though. As an example, this visualization gives an immediate notion of `distance' and even `angle'. But if you look carefully, our axioms for a vector space give no such notion. Just a formal way to do arithmetic on tuples of elements in a field.

Thus, we have a vector space and an underlying field which is usually the real or complex numbers. Our next task will be to stack adjectives onto our definition.  We look for, not only more adjectives to restrict our set under consideration, but also the {\it right} adjectives. For the purposes of composition operators, the set of vector spaces is just too large to get the sort of meaningful results we desire. We need to apply some restrictions to create more structure and better crystallize our studies. We start with a notion of distance by way of a norm which measures the `size' of a vector. Recall that for a complex number $z=a+bi$, its norm is given by $|z|^2=a^2+b^2$.

\begin{Definition}
A {\bf normed vector space} is a vector space $V$ over the field of real or complex numbers equipped with a norm, denoted by $|| \cdot||$, mapping vectors in $V$ to the real numbers, such that for any $u$ and $v$ in $V$ and $c$ in $\mathbb{R}$ or $\mathbb{C}$, we have

\begin{enumerate}
\item $||u||\geq 0$
\item $||u||=0$ \quad \text{if and only if} \quad $u=0$
\item $||c u||=|c| \cdot ||u||$
\item $||u+v|| \leq ||u||+||v||$ \vspace{5mm} \text{Triangle Inequality}
\end{enumerate}

\end{Definition} 

It's important to note that vector spaces do not necessarily have unique norms even though we tend to think of things as having one unique size. We must cast away our preconceived notions of `length' and understand that anything adhering to our set rules is by definition a norm. This will become immensely important when we take the abstract leap and try to measure the size of a function. This `chicken-and-egg' sort of development is an underlying feature of practically all of mathematics. The idea is that mathematicians have a concept in mind that they want to precisely define. They distill the key properties out and experiment with different properties, adding and scaling back, until a good working definition gives itself up. The fact that this process has worked so well is really a miracle.

As one example of the non-uniqueness of norms, take $\mathbb{R}^2$, the set of all pairs of real numbers $(a,b)$ which can be visualized via the $xy$-plane. Now consider the length of a vector from the origin to the point $(a,b)$. The Euclidean norm is given by $\sqrt{a^2+b^2}$. This is the length of a straight line, as the crow flies, from the origin to the point $(a,b)$. However, there are other useful norms as well. For example, consider the norm $|a|+|b|$. This is known as the taxicab norm. The intuition comes from imagining a perfectly gridded city where each block is one unit in length. Next imagine one starts at the origin and aims to arrive at integer coordinates $(a,b)$. Since one cannot (or rather shouldn't!) drive through a building, one is resigned to traveling a length of $|a|+|b|$ units to arrive at one's destination. As an exercise, the reader may try to show that this definition satisfies the properties of a norm.

Thus, now equipped with a normed vector space, we are able to utilize notions of distance. The original notion of a metric space, i.e. an abstract space with a notion of distance, was first brought to awareness due to the development of non-Euclidean geometries forcing the hand of mathematicians to make an effort to study what it means to measure something in these strange new worlds. 

We now have a vector space with added structure providing us with a way of measuring things. However, it turns out this still isn't enough. We need one more ingredient before we start really baking. And this time, we will need a bit of analysis. Now, the underlying objects of our consideration are going to be analytic functions of a complex variable. Analytic here means that the function can be represented as a power series. In calculus one learns that an infinitely-differentiable function about a point can be written as a power series. One of the many magical properties of complex functions is that if you can take one derivative, you get the rest for free, permitting one to express the function as a power series. This is the best deal you'll ever find at the supermarket.  Buy one, get $\infty$! BOGINF! \footnote{I am writing this from Lakeland, FL, hometown of the supermarket Publix, which is quite famous for its Buy One Get One (BOGO) deals.} This venture into the infinite expression of power series means we must work in the universe of analysis where limits reign supreme. 

Remember that closure property we had observed sneaking into our assumptions about vector spaces? It turns out it is important in the study of limits as well. If you've had a course in Analysis or an old-fashioned Calculus course where a bit of rigor still lingered, you might recall the following definition. 

\begin{Definition}
Let $f(x)$ be a function of a real variable that is well-defined on some interval about the point $x=a$, not necessarily including $a$. We say 

\begin{equation*}
\lim_{x \rightarrow a}f(x)=L
\end{equation*}
if for every real number $\epsilon>0$, there exists a real number $\delta>0$ such that $|f(x)-L|<\epsilon$ whenever $0<|x-a|<\delta$.
\end{Definition}

In informal language, this is just a precise way of saying that close points map to close points under the function $f$ near $a$. The same definition applies when you replace $x$ with $z$ and `interval' with a sufficiently small disk. 

Of importance to us is the idea of a limit applied to an infinite sequence of numbers. 

\begin{Definition}
We say a sequence $\{a_1, a_2, \cdots, \}$ converges to a limit $L$ if for every real number $\epsilon>0$, there exists a natural number $N$ such that $|a_n-L|<\epsilon$ whenever $n>N$.
\end{Definition}

Now, one aspect that is lacking here is that our definition assumes we have a guess as to what our limit is. What if we have no idea? Enter the Cauchy sequence.

\begin{Definition}
We say a sequence $\{a_1, a_2, \cdots, \}$ is a Cauchy sequence if for every real number $\epsilon>0$, there exists a natural number $N$ such that for all natural numbers $m,n>N$, we have $|a_n-a_m|<\epsilon$.
\end{Definition}

We thus have the following definition, guaranteeing our limits will stay in our space.
 
\begin{Definition}
We say a normed vector space is complete if every Cauchy sequence of points in the vector space converges to a limit that is also in the vector space. A complete normed vector space is known as a {\bf Banach space}.
\end{Definition}

Such vector spaces were first formalized and popularized by the Polish mathematician Stefan Banach in the early 20th century. He initially referred to them as spaces of type $B$ without mentioning why he chose the letter $B$. His surname comes from the surname of his mother given on his birth certificate who evidently disappeared from record when he was baptized only four days after he was born. His father seems to have sworn secrecy as to the identity of his mother. Nevertheless, her name will live on for centuries to come. Banach's life work in analysis has had a significant impact on analysis. His exposition Th\'eorie des Op\'erations lin\'eaires (written in French)\cite{Banach} \footnote{I had arranged to translate part of this text in order to satisfy the foreign language requirement for my PhD thesis. I procrastinated until my last year in graduate school at Purdue University when the requirement was dropped! Sometimes procrastination pays off. Usually not though.} is considered the founding text for operator theory and functional analysis, the arena in which composition operators find their natural home.

Next, we continue to stack structure onto our space. You'll recall that we observed the fact that norms contain no notion of angle. In order to introduce this, we have to add more structure to our norm. The most productive way of doing this is via the inner product. We define an inner product on a vector space $V$ as a function from the Cartesian product (think pairs) $V \times V \rightarrow \mathbb{C}$ satisfying the following criteria for $\lambda$ in our field and $u, v,w$ in $V$:

\begin{enumerate}
\item $\langle u, u \rangle \geq 0$
\item $\langle u, u \rangle=0$ \quad \text{if and only if} \quad $u=0$
\item $\langle u, v \rangle=\overline{\langle v, u \rangle} $
\item $\langle \lambda u, v \rangle=\lambda \langle u, v \rangle$
\item $\langle u+v, w \rangle=\langle u, w \rangle+\langle v, w \rangle$
\end{enumerate}

It turns out that every inner product space is a normed vector space but the converse is not necessarily true. Thus, for an inner product space we may define the angle between two vectors via the equation $\cos \theta=\frac {\langle u, v \rangle}{||u|| \cdot ||v||}$ which the reader may have seen in multivariate calculus course.

We define a {\bf Hilbert space} to be a Banach space whose norm gives rise to an inner product. Thus every Hilbert space is a Banach space but the converse is not necessarily true. The space upon which our composition operators will act, the Hardy space, is in fact a Hilbert space. Such spaces are named after great early 20th-century mathematician David Hilbert who dabbled in seemingly every corner of the mathematics of his time. He worked influentially in the field of mathematical foundations addressing the recent fall out of a mathematical revolution which we absolutely do not have space to address here. One important aspect of this revolution involved a shifting focus on {\it relations between objects} rather than objects themselves. At the time, it involved understanding that the axioms of a mathematical system need to be disentangled from our assumptions about what they ought to represent. Working in geometry, where the discovery of consistent non-Euclidean geometries led to a plurality of equally true geometric systems, he is quoted as saying

\begin{quote}
``{\it One must be able to say at all times--instead of points, straight lines, and planes--tables, chairs, and beer mugs.}''
\end{quote} 

I bring this up because all of the fuss about these abstract spaces really come to life when applying Hilbert's perspective to what the object in our space can be.

\subsection{The secret sauce}

\begin{quote}
``{\it Empires die, but Euclid’s theorems keep their youth forever.}"
\end{quote} 

The above quote was written by the Italian mathematician Vito Volterra on a postcard in reference to Mussolini's fascist Italy. \footnote{He was one of only 12 Italian professors out of 1,250 who refused to take a mandatory oath of loyalty to Mussolini's regime.} Historical records seem to indicate that he was responsible for an astounding conceptual innovation that hasn't aged a day and will play the key role as our special ingredient. Our special ingredient is not actually going to be a particular ingredient but rather the way in which we conduct our process, a concept that is well-known to bakers. It is the reason we spent so much time building abstract definitions in the last section. It will be the fire that breathes life into our subject matter, the composition operator.  We are going to endow our spaces with meaning that, I think, is a key player in the way humans conduct creative acts. 

You'll recall that for a Banach space, we need elements and a norm as foundational ingredients. Now we are accustomed to our elements being numbers and our norm being the classical distance between two points in Euclidean space given by the Pythagorean theorem. But our definition for our abstract spaces do not demand these interpretations. As Hilbert said, why not replace points and planes with tables and beer mugs! Rather than populating our normed vector space with classical vectors represented as arrows emanating from the origin to the ordered pair in the standard $xy$-plane, our vectors are going to be....functions.

Quite miraculously, as long as one is sufficiently careful, this approach works wonderfully! The idea that functions can be viewed as points \footnote{Vectors really, but when we do calculus, we routinely say points when we mean vectors.} in a strange space of infinite dimensions is very different from our standard notion of space and is the big idea Volterra stumbled upon in the late 1800's. We normally do not think much about the interior structure of a point and leave its interpretation vague. However, what riches can be attained when making the introspection! We are going to focus on one specific interpretation that is especially accomodating to our goals. In this interpretation, the points are not just any functions, but, chew slowly as it's a mouthful, rather functions of a complex variable that are analytic in the open unit disk which is defined to be the set of all points in the complex plane whose distance from the origin is strictly less than one. The unit disk turns out to be a powerful domain to work in, not just because of its symmetry, but also due to the celebrated Riemann Mapping Theorem which states that an open connected proper subset of the complex plane is conformal to the unit disk where conformal essentially means angle and orientation preserving. 

Now, recall that a very special feature of analytic functions is that they admit a power series. That is, given an analytic function $f$, one may write it as 

\begin{equation}
f(z)=\sum_{k=0}^{\infty}a_k z^k
\end{equation}
for some complex coefficients $\{a_k\}_0^{\infty}$. Now, you might be wondering, how on Earth do you find the length of a vector represented by a function? For the Hardy space, we use a norm that resembles the Pythagorean theorem. It turns out that for a Hilbert space, inner products are related to norms via the formula $\langle u, u \rangle=||u||^2$. We define a norm on the Hardy space $H^2(\mathbb{D})$ via 

\begin{equation}
||f||_{H^2(\mathbb{D})}^2=\sum_{k=0}^{\infty}|a_k|^2.
\end{equation}

If $f$ has finite norm, that is $||f||_{H^2(\mathbb{D})}^2<\infty$, we define it to be in the Hardy space. Additionally, it follows that our inner product for two functions $f(z)=\sum_{k=0}^{\infty}a_k z^k$ and $g(z)=\sum_{k=0}^{\infty}b_k z^k$ analytic in the unit disk  is given by 

\begin{equation}
\langle f, g \rangle_{H^2(\mathbb{D})}^2=\sum_{k=0}^{\infty}a_k \overline{b_k}.
\end{equation}

We will omit the subscript on the inner product when understood from context. Thus, we have our definition. It turns out that this definition is equivalent to the one given in the introduction if one sets $p=2$. 

\begin{proof}
\begin{align}
||f||_{H^2(\mathbb{D})}^2&=\sup_{0<r<1} \int_0^{2 \pi}|f(re^i \theta)|^2 \frac{d\theta}{2 \pi}\\
&=\sup_{0<r<1} \int_0^{2 \pi}\left(\sum_{j=0}^{\infty}a_jr^je^{ij\theta}\right)\left(\sum_{k=0}^{\infty}\overline{a_k}r^ke^{-ik\theta}\right)\frac{d\theta}{2 \pi}\\
&=\sup_{0<r<1} \int_0^{2 \pi}\sum_{j=0}^{\infty}\sum_{k=0}^{\infty}a_j\overline{a_k}r^{j+k}e^{i(j-k)\theta}\frac{d\theta}{2 \pi}\\
&\sup_{0<r<1}\sum_{j=0}^{\infty}\sum_{k=0}^{\infty}a_j\overline{a_k}r^{j+k}\int_0^{2 \pi}e^{i(j-k)\theta}\frac{d\theta}{2 \pi}\\
&=\sum_{k=0}^{\infty}a_k\overline{a_k}r^{2k}=\sum_{k=0}^{\infty} |a_k|^2
\end{align}
\end{proof}

It turns out that the Hardy space is a concrete (and fruitful!) example of a Hilbert space. 

Thus, to recap, we have a a space defined by special features that endow it with arithmetic operations and a notion of length for each element. We then made the conceptual leap of substituting functions in for our elements, and it works! Doing analysis in such a domain gives rise to the field of functional analysis.

Furthermore, the Hardy space is an example of a special type of Hilbert space known as a reproducing kernel Hilbert space (RKHS). 

\begin{Definition}
We say a Hilbert space is a reproducing kernel Hilbert space (RKHS) if the following holds. Suppose we have a Hilbert space $H$ of real or complex-valued functions equipped with pointwise addition and scalar multiplication and a set $S$ (often the unit disk). We define the evaluation functional $L_z$ on $H$ as the map $L_z : f \rightarrow f(z)$ for all $f \in H$. If this evaluation functional is continuous for every value $z$ in $S$, we say that $H$ is a reproducing kernel Hilbert space.
\end{Definition}

Now, and this is important, there is a very powerful theorem known as the Riesz representation theorem that says the following. For each RKHS there is a special unique and privileged function $k_z$ in $H$ called the kernel function that completely determines the RKHS where for every function $f$ in the Hilbert space $H$, we have

\begin{equation}
f(z)=\langle f, k_z \rangle.
\end{equation}

Now this is a rather abstract concept so it would serve us well to take a moment and explicitly compute the kernel for our Hardy space.

\begin{Theorem}
The kernel function for the Hardy space $H$ is given by 
\begin{equation}
k_w(z)=\frac{1}{1-\overline{w}z}.
\end{equation}
\end{Theorem}

\begin{proof}
Recalling the definition of our inner product in the Hardy space, we must demonstrate that it satisfies the reproducing property. Suppose $f(z)=\sum_{k=0}^{\infty}a_k z^k$ and recall that for $\overline{w}z<1$ the geometric series gives us $\frac{1}{1-\overline{w}z}=\sum_{k=0}^{\infty}\left(\overline{w}z\right)^k$. We then have 
\begin{equation}
\langle f, k_w \rangle=\left\langle \sum_{k=0}^{\infty}a_k z^k, \sum_{k=0}^{\infty}\overline{w}^kz^k \right\rangle= \sum_{k=0}^{\infty}a_k w^k=f(w)
\end{equation}
as desired.
\end{proof}

These kernel functions play a critical role in the study of RKHS's for a variety of reasons that would take us too far afield in the present moment. Let's proceed to the seasoning before arriving at our destination.

\subsection{A dash of infinite-dimensional linear algebra}

Vector spaces are a natural playground for the subject of linear algebra. Studying finite dimensional versions in depth are part of every math departments core curriculum. The applications in STEM are endless. When asked what topic a given student should learn more deeply, linear algebra is always my go-to answer. It just shows up everywhere in massive quantities. It's been said that all of machine learning is just glorified linear algebra.\footnote{One ought to be careful here, however. While all the math behind AI is assumed linear from the outset, the neural networks in our brains are, in fact, {\it not} linear in nature, a substantial deviation from machines.} Now, for essentially every concept in finite vector spaces, there is an analog in infinite-dimensional vector spaces. Our focus will be specifically on the Hardy space although other vector spaces in functional analysis are similar in nature. 

Let's begin by discussing what our basis is. Recall that a basis for a vector space is a linearly independent spanning subset, meaning a minimal set of elements in the vector space such that any element in the vector space can be written as a linear combination of those elements. While there are all sorts of options for choosing your basis, the number of basis elements, known as the dimension, remains invariant. For the Hardy space, it turns out that there is a natural way of identifying this set. Since analytic functions are guaranteed a power series expansion, we take the set of monomials $\{1, z, z^2,...\}$ as our basis. Furthermore, using our definition above, a brief calculation shows that this basis is orthonormal meaning $\langle z^m, z^n \rangle=\delta_{mn}$ where $\delta_{mn}$ is the Kronecker delta function, equaling $0$ if $m \neq n$ and $1$ if $m=n$. Thus a vector in the Hardy space representing the `point' $f(z)=\sum_{k=0}^{\infty}a_kz^k$ can be written as $(a_0, a_1, a_2,...)^T$ where $T$ denotes the transpose. We call infinite-dimensional matrices acting on these infinite-dimensional vectors {\it operators}. The study of these objects is called operator theory. Think infinite-dimensional linear algebra.

In finite-dimensional linear algebra, every linear transformation can be represented by a matrix relative to some basis. Experience has shown with great emphasis that certain features of a given matrix are of particular importance. These are things like the norm of a matrix, its eigenvalues and eigenvectors, it's inverse if it exists, its adjoint, the determinant, and many others. Each of these has an infinite-dimensional counterpart that does not always exactly coincide withe the finite-dimensional definitions.

\section{PUT IT IN THE OVEN ALREADY!}

Now we finally arrive at the point where we can define our composition operators in the rich context they deserve. Once all the hard work is done and you put it in the oven, you just watch it bake. With the appropriate preparation now complete and without further ado, we present the composition operator.

\begin{Definition}
Let $H$ be a RKHS and $f \in H$. A composition operator $C_{\phi}$ induced by symbol $\phi$ is the operator defined by 

$$C_{\phi}f=f \circ \phi.$$

\end{Definition} 

That's it. Upon first glance, it looks like something learned in introductory algebra, yet the context given is rich and it continues to bear fruit. Note that for a given composition operator, the map $\phi$ is fixed while the vectors $f$ vary. It's also important to note that $\phi$ must be a self-map of the unit disk in order for the composition operator to be well-defined. 

Note that composition operators are linear since, for a scalar $c$, we have

\begin{align*}
    C_{\phi}(f+g) &=(f+g) \circ \phi=f\circ \phi+g\circ \phi =C_{\phi}f+C_{\phi}g \\
    C_{\phi}(cf) &= cf \circ \phi=cC_{\phi}f 
\end{align*}

The fact that $C_{\phi}$ is linear means that it has an (infinite) matrix representation. To observe this in action, let's take $\phi(z)=z^2$ and determine its matrix representation. This is a self-map of the unit disk as $|z|<1$ certainly implies $|z|^2<1$. Hence for $f(z)=\sum_{k=0}^{\infty}a_kz^k$ we have the following:

\begin{equation}
C_{z^2}f (z)=f (z^2)=\sum_{k=0}^{\infty}a_k(z^2)^k=a_0+a_1z^2+a_2z^4+a_3z^6+\cdots
\end{equation}
and hence, written in the standard basis given by $\{1, z, z^2, ...\}$, we have

\begin{equation}
C_{z^2}f=
\begin{pmatrix}
1 & 0 & 0 & 0 & \cdots \\
0 & 0 & 0 & 0 & \cdots \\
0 & 1 & 0 & 0 & \cdots \\
0 & 0 & 0 & 0 & \cdots \\
0 & 0 & 1 & 0 & \cdots \\
\vdots & \vdots & \vdots & \vdots & \ddots \\
\end{pmatrix}
\begin{pmatrix}
a_0 \\ a_1 \\ a_2 \\ a_3 \\ a_4 \\  \vdots
\end{pmatrix}=
\begin{pmatrix}
a_0 \\ 0 \\ a_1 \\ 0 \\ a_2 \\ \vdots
\end{pmatrix}
\end{equation}

Now usually matrix representations for composition operators are too complex to be useful and more function theoretic approaches are taken. Nonetheless, it is instructive to understand composition operators as analogs of matrices in our infinite-dimensional space. It's also a prompt for us that we can ask the same sort of questions that one asks in linear algebra.

\section{A NEW WORLD TO EXPLORE}

Finally, with composition operators in hand, we briefly review a few fun highlights, underlining the connection with finite-dimensional linear algebra concepts. Context matters as well. Results vary contingent upon which vector space is used with the Hardy space being, arguably, the most celebrated.

The first order of business that finite-dimensional bakers need not worry about is that of boundedness. The idea is that we want our operator to map bounded sets to bounded sets. This comes for free for matrices in finite dimensions. We define boundedness in our context.

\begin{Definition}
An operator $T$ acting on a RKHS $H$ is call {\bf bounded} if there exists an $M>0$ such that for all $f$ in $H$, 
\begin{equation}
||Tf||_{H} \leq M ||f||_{H}
\end{equation}
\end{Definition}

The subscript $H$, which is often omitted, is there to remind us that this is the norm given on the RKHS. In our case, the Hardy space, it is $||f||_{H^2(\mathbb{D})}=\sum_{k=0}^{\infty}|a_k|^2$. This also leads to a way to find the `size' of an operator. Just as one may define a norm on a matrix (recall there are numerous definitions), we define the norm of an operator as follows.

\begin{Definition}
The norm of a bounded operator $T$ acting on a RKHS $H$ is given by 
\begin{equation}
||T||_{{\it op}} =\sup_{||f||_{H} \leq 1}||Tf||_{H}
\end{equation}
\end{Definition}

Notice that there are a lot of different norms floating around. Be careful. Additionally and perhaps surprisingly, it turns out that the boundedness of an operator is equivalent to continuity! 

One of the most important results in the theory of composition operators was actually proven over $100$ years ago in 1925 by the English mathematician J.L. Littlewood, a close collaborator of G.H. Hardy for whom our space is named. It is called Littlewood's subordination principle. Using the integral definition of the Hardy space, Littlewood showed that

\begin{equation}
\int_0^{2\pi} |f(\phi(re^{i\theta}))|^2 d\theta \leq \int_0^{2\pi} |f(re^{i\theta})|^2 d\theta
\end{equation}
when $\phi$ is an analytic self-map of the unit disk. If one takes a good hard look at that, it says something rather profound. Every analytic self-map of the unit disk induces a bounded composition operator on the Hardy space $H^2(\mathbb{D})$.\footnote{Littlewood actually proved it for $H^p(\mathbb{D})$ for $1 \leq p <\infty$.}

Addressing the question of eigenvalues actually marks the start of the study of composition operators. Dating all the way back to 1871, Ernst Schr\"{o}der studied the functional equation $f(\phi(z))=\lambda f(z)$. In our context, we may interpret it as the eigenvalue equation $C_{\phi}f=\lambda f$ which is just a infinite-dimensional relabeling of the formula $Ax=\lambda x$. In such a case, it is not sufficient to find the formal solution. One must ensure such solutions are in the appropriate space as well. 

Now for a linear map in finite dimensions, solving $Ax=\lambda x$ is equivalent to determining when $A-\lambda I$ fails to be invertible, where $I$ is the identity matrix. This does not carry into infinite-dimensions. One may construct scenarios where $A-\lambda I$ is not invertible yet $\lambda$ is not an eigenvalue.

\begin{Definition}
Given an operator $T$ acting on a RKHS $H$ , the {\bf spectrum} is defined to be 
\begin{equation}
\sigma(T)=\{ \lambda \in \mathbb{C} \mid T-\lambda I \hspace{2mm} \text{is not invertible} \}
\end{equation}
\end{Definition}

The set of eigenvalues is known as the point spectrum, a subset of the spectrum. Determining the spectrum of a composition operator is no small task and is highly dependent on the underlying map $\phi$. There is also a beautiful and intricate relationship between the spectrum and the fixed-point behaviour of $\phi$ but that would take us too far afield. The author encourages the reader to explore further through her own investigations or researching the literature.

To find the inverse of a composition operator, one must ensure that it exists. In such a case, the map $\phi$ must be bijective, that is, one-to-one and onto. It turns out that these are precisely the class of maps known as the automorphisms of the unit disk and can be written in the form

\begin{equation}
\phi(z)=e^{i \theta} \frac{a-z}{1-\overline{a}z}
\end{equation}
where $|a|<1$ and $\theta$ is a real number. In such a case, one may use the definition of composition operator to verify that $(C_{\phi} \circ C_{\phi^{-1}})(f)=(C_{\phi^{-1}} \circ C_{\phi})(f)=f$ from which it follows that $(C_{\phi})^{-1}=C_{\phi^{-1}}$.

Finally, recall the adjoint of a matrix $A$ is defined to be the matrix $A^*$ such that $\langle Ax, y \rangle = \langle x, A^*y \rangle$. The same definition applies in infinite dimensions. As expected, the adjoint depends highly on the nature of the underlying map $\phi$.  One beautiful identity is immediate. Recalling our kernel functions $k_w$ and their special reproducing property, we have

\begin{equation}
 \langle f, C_{\phi}^* k_w \rangle=\langle C_{\phi}f, k_w \rangle = \langle f \circ \phi, k_w \rangle= f(\phi(w))=\langle f, k_{\phi(w)} \rangle
\end{equation}
and hence we attain the wonderful little formula that could fit in your pocket:

\begin{equation}
C_{\phi}^*(k_w)=k_{\phi(w)}.
\end{equation}

Similar to composition operators acting on the Hardy space, we also have multiplication operators, known as Toeplitz operators. These are defined by $T_{\phi}(f)=P(\phi f)$ where $P$ denotes the projection orthogonally onto $H^2(\mathbb{D})$. You might be thinking it should just be $T_{\phi}(f)=\phi f$. However, multiplying the function $f$ in $H^2(\mathbb{D})$ by $\phi$ might knock it out of $H^2(\mathbb{D})$, hence we project back onto our space. In practice, this just means writing out the infinite series of the product at eliminating the terms with negative Fourier coefficients, i.e. terms that look like $a_nz^n$ where $n$ is a negative integer.

As a final glaze on our product before we see ourselves out, there is a wonderful adjoint formula due to Cowen \cite{C} for when $\phi$ is a linear fractional map. defined by 

\begin{equation}
f(z)=\frac{az+b}{cz+d}
\end{equation}
where the coefficients $a$, $b$, $c$, and $d$ are complex numbers such that $ad-bc \neq 0$ (otherwise $f(z)$ is a constant function). We have the following proof in the same manner it was first given by Cowen in 1988.

\begin{Theorem}
Let 
\begin{equation}
\phi(z)=\frac{az+b}{cz+d}
\end{equation}

be a linear fractional self-map of the unit disk. Then the adjoint $C_{\phi}^*=T_gC_{\sigma}T_h^*$ where $g(z)=\frac{1}{-\overline{b}z+\overline{d}}$, $\sigma(z)=\frac{\overline{a}z-\overline{c}}{-\overline{b}z+\overline{d}}$ and $h(z)=cz+d$.
\end{Theorem}

\begin{proof}
It turns out that the kernel functions are dense in the Hardy space. What this means for us by some technical reasons is that it is sufficient to prove the theorem when $f(z)=k_w(z)$. Thus, what we are really looking for is $C_{\phi}^*k_w=k_{\phi(w)}$. In the same way we showed this formula, a small calcuation also shows that $T_h^*k_w=\overline{h(w)}k_w$ when $h$ is a multiplier (meaning $hf$ is in the Hardy space for all $f$ in the Hardy space). Our $g$ and $h$ can be shown to be multipliers. Additionally, one can show that $\sigma$ is a self-map of the unit disk. Hence the formula is well-defined. Then a calculation shows

\begin{align}
T_gC_{\sigma}T_h^*(k_w)(z)&=\overline{h(w)}T_gC_{\sigma}(k_w)(z)\\
&=\overline{(cw+d)}\left(\frac{1}{-\overline{b}z+\overline{d}}\right)\left(\frac{1}{1-\overline{w}}\frac{\overline{a}z-\overline{c}}{-\overline{b}z+\overline{d}}\right)\\
&=\frac{\overline{cw}+\overline{d}}{-\overline{b}z+\overline{d}-\overline{wa}z+\overline{wc}}\\
&=\frac{1}{1-\overline{\phi(w)}z}=k_{\phi(w)}=C_{\phi}^*(k_w)(z).
\end{align}

\end{proof}

While there is so much more to say, alas, our time together must draw to a close. We have already meandered farther than we were supposed to. It is hoped that the reader takes this tasty introduction to the idea behind a composition operator as the starting line to a grand adventure filled with inquiry and discovery. There are many interesting operators and currently a flurry of activity is transpiring in the theory of such operators acting on a generalized Hardy space known as the Drury-Arveson space, where the underlying vectors are functions of several complex variables. Much remains to be explored. Let's watch what the future brings. Better yet, let's engage and aim to play a role in its development. Until next time.

\end{document}